\documentclass[english]{amsart}
\usepackage[T1]{fontenc}
\usepackage[latin9]{inputenc}
\usepackage{float}
\usepackage{textcomp}
\usepackage{amstext}
\usepackage{amsthm}
\usepackage{amssymb}
\usepackage{graphicx}

\makeatletter

\newcommand{\noun}[1]{\textsc{#1}}
\providecommand{\tabularnewline}{\\}

\numberwithin{equation}{section}
\numberwithin{figure}{section}
\theoremstyle{definition}
\newtheorem*{example*}{\protect\examplename}
\theoremstyle{plain}
\newtheorem*{prop*}{\protect\propositionname}
\theoremstyle{definition}
\newtheorem*{defn*}{\protect\definitionname}
\theoremstyle{plain}
\newtheorem*{thm*}{\protect\theoremname}
\theoremstyle{plain}
\newtheorem*{cor*}{\protect\corollaryname}
\theoremstyle{plain}
\newtheorem*{lem*}{\protect\lemmaname}

\makeatother

\usepackage{babel}
\providecommand{\corollaryname}{Corollary}
\providecommand{\definitionname}{Definition}
\providecommand{\examplename}{Example}
\providecommand{\lemmaname}{Lemma}
\providecommand{\propositionname}{Proposition}
\providecommand{\theoremname}{Theorem}

\begin{document}
\global\long\def\lm{\lim\nolimits}%
\global\long\def\it{\operatorname{int}\nolimits}%
\global\long\def\ad{\operatorname{adh}\nolimits}%
\global\long\def\cl{\operatorname{cl}\nolimits}%
\global\long\def\ih{\operatorname{inh}\nolimits}%
\global\long\def\card{\mathrm{card\,}}%

\global\long\def\0{\mathrm{\varnothing}}%

\title{Acquiring a dimension:\linebreak{}
 from topology to convergence theory}
\author{Szymon Dolecki}
\address{Institut de Mathématiques de Bourgogne, Université de Bourgogne et
Franche Comté, Dijon, France}
\date{\today}
\begin{abstract}
Convergence theory is an extension of general topology. In contrast
with topology, it is closed under some important operations, like
exponentiation. With all its advantages, convergence theory remains
rather unknown. It is an aim of this paper to make it more familiar
to the mathematical community.
\end{abstract}

\maketitle

\section{Introduction}

Sometimes, a change of perspective reveals unexpected prospects. This
was the case of the emergence of imaginary numbers within attempts
of solving algebraic equations with real coefficients. A revealed
universe of complex numbers had a new, metaphysical dimension, that
of imaginary numbers.

Thinking of topology as of convergence of filters, rather than in
terms of families of open sets or alike, was another case. Convergence
theory emerged as a universe, hosting the planet of general topology.
From this new perspective, topology appears laborious and austere.

In both cases, that of complex numbers and this of convergence theory,
unknown objects materialized. They were often solutions of quests
that were unconceivable in the old contexts. But also, surprisingly
luminous solutions to tough problems manifested within the original
framework. For instance, complex numbers are prodigiously instrumental
in the theory of linear differential equations.

Convergence theory entirely redesigns the concepts of compactness
and completeness, covers, hyperspaces, quotient and perfects maps,
regularity and topologicity. It elucidates the role of sequences and
that of function spaces. From this broader viewpoint, various topological
explorations unveil their essential aspects and can be apprehended
more consciously, and thus often more fruitfully.

In 1948 \noun{G. Choquet} published a foundational paper \cite{cho},
in which he noticed that the so-called \emph{Kuratowski limits} cannot
be topologized. From today's perspective, they are dual convergences
on \emph{hyperspaces} of closed sets, which in general are not topological,
because the class of topologies is not \emph{exponential} (\footnote{In other terms, the category of topologies with continuous functions
as morphisms is not Cartesian-closed.}). The introduction of \emph{pseudotopologies} by \noun{Choquet} marked
a turning point in our perception of general topology.

In this paper, I wish to present some salient traits of this theory,
which fascinates me, and to transmit some of my enthusiasm. If you
have a fancy, you may have a look at an introductory paper \cite{dolecki.BT},
at a textbook \cite{CFT}, and soon at a forthcoming book \cite{SD_Royal}.

\section{The choice of filters}

Filters arise naturally, when one considers convergence, for example,
of sequences in the real line. In order to fix notation, and to introduce
a new perspective, we say that a sequence $(x_{n})_{n}$ of elements
of the real line $\mathbb{R}$ \emph{converges} to $x$ whenever for
each $\varepsilon>0$, the set $\{n:x_{n}\notin(x-\varepsilon,x+\varepsilon)\}$
is finite.

A set $V$ is a \emph{neighborhood} of $x$ if there is $\varepsilon>0$
such that $(x-\varepsilon,x+\varepsilon)\subset V$. The family of
all neighborhoods of $x$ is denoted by $\mathcal{N}(x)$. Let
\begin{equation}
\mathcal{E}:=\{E\subset\mathbb{R}:\{n:x_{n}\notin E\}\textrm{ is finite}\}.\tag{sequential filter}\label{eq:sequential}
\end{equation}
We notice that a sequence $(x_{n})_{n}$ converges to a point $x$
if and only if
\begin{equation}
\mathcal{N}(x)\subset\mathcal{E}.\label{eq:conv_filt}
\end{equation}

A non-empty family $\mathcal{F}$ of subsets of a given non-empty
set $X$ is called a \emph{filter} on $X$ if
\[
(F_{0}\in\mathcal{F})\wedge(F_{1}\in\mathcal{F})\Longleftrightarrow F_{0}\cap F_{1}\in\mathcal{F}.
\]
A filter $\mathcal{F}$ is \emph{proper} if $\0\notin\mathcal{F}$
(\footnote{A unique \emph{improper filter} on $X$ is $2^{X}$, the family of
all subsets of $X$.}).

Of course, $\mathcal{N}(x)$ is a proper filter for each $x\in\mathbb{R}$,
and $\mathcal{E}$ is a proper filter for each sequence $(x_{n})_{n}$.
Moreover, convergence of a sequence to a point is characterized by
the inclusion of filters (\ref{eq:conv_filt}).

Notice that in the description of convergence of sequences (\footnote{Here, on the real line, but it is valid in any topological space.}),
the order on the set of indices has vanished for a good reason that,
from the convergence viewpoint, it is irrelevant. A permutation of
indices of a sequence has no impact on convergence. Moreover, convergence
of a sequence is invariant under arbitrary finite-to-one transformations
of the set of its indices.

As we have seen in the definition of the filter $\mathcal{E}$ associated
with a sequence, what counts are \emph{cofinite} (that is, having
finite complements) subsets of indices. And ultimately, each convergence
relation is reduced to a comparison of appropriate filters.

The framework of filters appears best suited to formalize convergence.
Other attempts, for instance using \emph{net}s (\footnote{The so-called \emph{Moore-Smith convergence}.})
\cite{kelley}, have serious drawbacks. 

The set of filters on a given set $X$, ordered by inclusion, is a
complete lattice, in which extrema are easily described with the aid
of the (bigger) lattice of all (isotone) families of subsets of $X$.
Restricted to proper filters, this order is no longer an upper lattice,
but it disposes of a huge set of maximal elements, called \emph{ultrafilters}.

\section{Convergences}

Given a non-empty set $X$, a \emph{convergence} $\xi$ is any relation
between (non-degenerate) filters $\mathcal{F}$ on $X$ and points
of $X$ written
\[
x\in\lm_{\xi}\mathcal{F}\text{,}
\]
under the provision that $\mathcal{F}_{0}\subset\mathcal{F}_{1}$
implies $\lm_{\xi}\mathcal{F}_{0}\subset\lm_{\xi}\mathcal{F}_{1}$
(\footnote{For every filters $\mathcal{F}_{0}$ and $\mathcal{F}_{1}$ on $X$.}),
and $x\in\lm_{\xi}\{x\}^{\uparrow}$ for each $x\in X$, where $\{x\}^{\uparrow}$
stands for the \emph{principal ultrafilter }of $x$ (\footnote{$\{x\}^{\uparrow}:=\{A\subset X:x\in A\}$.}). 

In the convergence framework, \emph{topologies} form an important
subclass. Of course, a filter $\mathcal{F}$ converges to a point
$x$ in a topological space, whenever each open set $O$ containing
$x$ belongs to $\mathcal{F}$, in other words, whenever $\mathcal{N}(x)\subset\mathcal{F}$,
as in the particular case discussed at the beginning. 

\emph{Continuity} is what one would expect. If $\xi$ is a convergence
on $X$, and $\tau$ is a convergence on $Y$, then $f\in Y^{X}$
is \emph{continuous}
\[
f\in C(\xi,\tau)
\]
whenever $x\in\lm_{\xi}\mathcal{F}$ entails $f(x)\in\lm_{\tau}f[\mathcal{F}]$
(\footnote{Where $f[\mathcal{F}]:=\{f(F):F\in\mathcal{F}\}.$ This family is
not necessarily a filter on $Y$, as $f$ need not be surjective,
but is a \emph{base} of a filter; a subfamily $\mathcal{B}$ is said
to be a base of a filter $\mathcal{G}$ if for each $G\in\mathcal{G}$
there is $B\in\mathcal{B}$ such that $B\subset G$. A convergence
is obviously extended to filter-bases by $\lm\mathcal{B}=\lm\mathcal{G}$
if $\mathcal{B}$ is a base of $\mathcal{G}$.}) for every filter $\mathcal{F}$ on $X$. Other basic constructions,
known from topology, are based on the concept of continuity. 

Primarily, a convergence $\zeta$ is \emph{finer} than a convergence
$\xi$ ($\zeta\geq\xi$) whenever the identity map $i$ is continuous
$i\in C(\zeta,\xi)$ (\footnote{The finest convergence on $X$ is the discrete topology $\iota$,
for which $x\in\lm_{\iota}\mathcal{F}$ implies that $\mathcal{F}=\{x\}^{\uparrow}$;
the coarsest one is the chaotic topology $o$, that is, $X=\lm_{o}\mathcal{F}$
for each filter on $X$. }). The set of convergences on a given set, is a complete lattice,
in which the extrema admit very simple formulae (\footnote{$\lm_{\bigvee\Xi}\mathcal{F}=\bigcap_{\xi\in\Xi}\lm_{\xi}\mathcal{F},$
and $\lm_{\bigwedge\Xi}\mathcal{F}=\bigcup_{\xi\in\Xi}\lm_{\xi}\mathcal{F}.$}).

The \emph{initial convergence} $f^{-}\tau$ is the coarsest convergence,
for which $f\in C(f^{-}\tau,\tau)$. The \emph{final convergence}
$f\xi$ is the finest convergence on the codomain of $f$, for which
$f\in C(\xi,f\xi)$. A product convergence $\prod_{j\in J}\theta_{j}$
is, of course, defined as $\bigvee_{j\in J}p_{j}^{-}\theta_{j}$,
where $p_{j}$ is the projection from the product set $\prod_{i\in J}X_{i}$
onto the $j$-th component $X_{j}$ carrying the convergence $\theta_{j}$.
Alike for other operations.

An elementary, though non-trivial example of non-topological convergence
is 
\begin{example*}
The \emph{sequential modification} $\mathrm{Seq}\,\nu$ of a usual
topology of the real line $\nu$, in which $x\in\lm_{\mathrm{Seq}\,\nu}\mathcal{F}$
provided that there exists a sequence $(x_{n})_{n}$ such that $\lim_{n\rightarrow\infty}x_{n}=x$,
and $\{x_{k}:k>n\}\in\mathcal{F}$ for each $n$ (\footnote{Mind that $\mathcal{F}$ need not be equal to (\ref{eq:sequential})
defined by that sequence.}). In other terms, a filter converges to a point whenever it is \emph{finer}
than a \emph{sequential filter }converging to that point. However,
there exists no coarsest filter converging to a given point. In fact,
the infimum of all sequential filters converging to $x$ is the neighborhood
filter $\mathcal{N}_{\nu}(x)$ of $x$ for the usual topology, but
$x\notin\lm_{\mathrm{Seq}\,\nu}\mathcal{N}_{\nu}(x)$, because each
$V\in\mathcal{N}_{\nu}(x)$ is uncountable (\footnote{Hence, is not of the form $\{x_{k}:k>n\}$.}). 
\end{example*}
A collection $\mathbb{D}$ of filters converging to $x$ for a convergence
$\xi$ is called a \emph{pavement} of $\xi$ at $x$ whenever if $x\in\lm_{\xi}\mathcal{F}$
then there is $\mathcal{D}\in\mathbb{D}$ such that $\mathcal{D}\subset\mathcal{F}$.
The \emph{paving number} $\mathfrak{p}(x,\xi)$ is the least cardinal
such there is a pavement of $\xi$ at $x$ of cardinality $\mathfrak{p}(x,\xi)$.

In the case of topological convergences, the paving number is always
equal to $1$. In the example above, it is infinite (\footnote{It can be shown that this convergence is not \emph{countably paved},
that is, $\mathfrak{p}(x,\mathrm{Seq}\,\nu)>\aleph_{0}$ for each
$x$.}).

\begin{figure}[H]
\centering{}\includegraphics[scale=0.25]{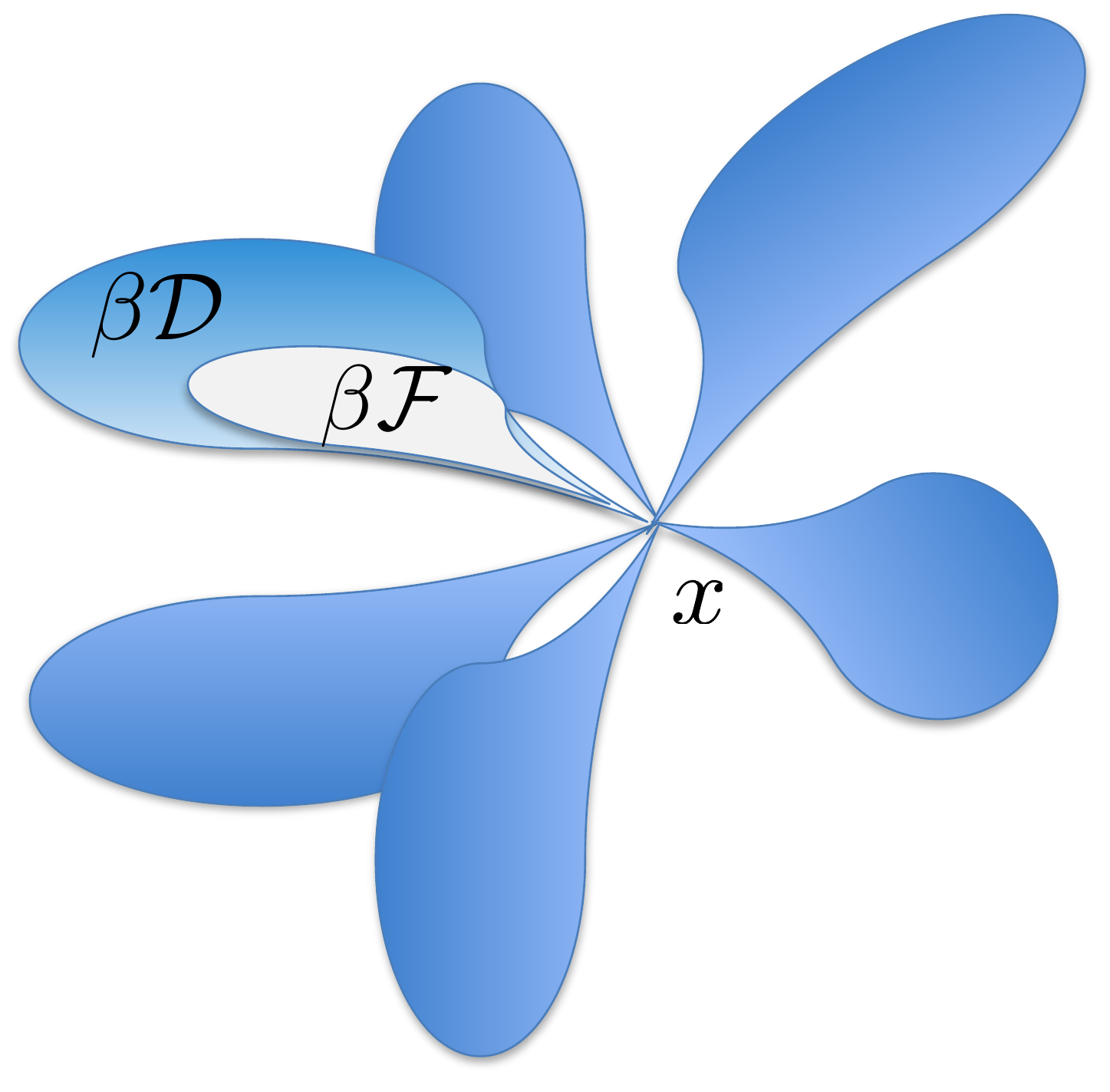} \caption{Here, a filter $\mathcal{F}$ finer than an element $\mathcal{D}$
of a pavement at $x$, hence $\mathcal{F}$ converges to $x$. By
$\beta\mathcal{F}$ we denote the set of ultrafilters that are finer
than a filter $\mathcal{F}$. Accordingly, $\mathcal{F}\protect\geq\mathcal{D}$
($\mathcal{F}$ is finer than $\mathcal{D}$) if and only if $\beta\mathcal{F}\subset\beta\mathcal{D}$.}
\end{figure}

\section{From topologies to pretopologies}

Pretopologies constitute a first generalization of topologies, and
have already been considered, under various names, by \noun{Sierpi\'{n}ski,
\v{C}ech, Hausdorff}, and\noun{ Choquet}. 

On one hand, pretopologies include topologies, which has been so far
the most known and studied class of convergences. They have many similarities
with topologies, and one difference: the \emph{ adherence}, an analogue
of \emph{topological closure}, is in general not idempotent. On the
other hand, the class of pretopologies has a much simpler structure
than the class of topologies.

A convergence is a \emph{pretopology} if for each point (of the underlying
set), there is a coarsest filter converging to that point. Therefore,
a convergence is a pretopology if and only if its paving number is
$1$. 

For an arbitrary convergence $\theta$, the filter
\[
\mathcal{V}_{\theta}(x):=\bigcap\{\mathcal{F}:x\in\lm_{\theta}\mathcal{F}\}
\]
is called the \emph{vicinity filter} of $\theta$ at $x$. 
\begin{prop*}
A convergence $\xi$ is a pretopology if and only if $x\in\lm_{\xi}\mathcal{V}_{\xi}(x)$
for each $x$ in the underlying set $\left|\xi\right|$.
\end{prop*}
Consequently, each topology $\xi$ is a pretopology, and if it is,
then $\mathcal{V}_{\xi}(x)=\mathcal{N}_{\xi}(x)$ is the \emph{neighborhood
filter} of $\xi$ at $x$.

If $\mathcal{A}$ is a family of subsets of $X$, then the \emph{grill
}of $\mathcal{A}$ is defined by
\[
\mathcal{A}^{\#}:=\bigcap\nolimits _{A\in\mathcal{A}}\{H\subset X:A\cap H\neq\0\}.
\]

For an arbitrary convergence $\theta$, the \emph{adherence} $\ad_{\theta}A$
of a set $A$ can be defined by
\begin{equation}
x\in\ad_{\theta}A\Longleftrightarrow A\in\mathcal{V}_{\theta}(x)^{\#}.\tag{set-adherence}\label{eq:adh}
\end{equation}

For any convergence $\theta$, a set $A$ is called $\theta$\emph{-closed}
if $\ad_{\theta}A\subset A$. The $\theta$-closure is defined by
$\cl_{\theta}A:=\bigcap_{H\supset A}\{H:\ad_{\theta}H\subset H\}$.

Notice that if $\xi$ is a pretopology, then by (\ref{eq:adh}) its
adherence determines all its vicinity filter, hence its convergent
filters.
\begin{prop*}
A pretopology is a topology if and only if its adherence is idempotent.
\end{prop*}
In other terms, a pretopology $\xi$ is topological whenever $\ad_{\xi}(\ad_{\xi}A)\subset\ad_{\xi}A$
for each $A$. In this case, $\ad_{\xi}A$ is, of course, $\xi$-closed
and and equal to $\cl_{\xi}A$, the closure of $A$.

\emph{Topologies, pretopologies}, and many other fundamental classes
of convergences are \emph{projective}. This means that for each convergence
$\theta$, there exists a finest topology $J\theta$ among the topologies
that are coarser than $\theta$. It is easy to see that so defined
$J$ is \emph{concrete }($\left|J\theta\right|=\left|\theta\right|$),
\emph{increasing }($\theta_{0}\leq\theta_{1}$ implies $J\theta_{0}\leq J\theta_{1}$),
\emph{idempotent }($J(J\theta)=J\theta$), as well as \emph{descending}
($J\theta\leq\theta$) (\footnote{for each convergences $\theta,\theta_{0},\theta_{1}$}).
If $J$ preserves continuity, that is, if $C(\xi,\tau)\subset C(J\xi,J\tau)$
for any convergences $\xi$ and $\tau$, then $J$ is a \emph{concrete
functor} (\footnote{Basic facts from category theory are used here instrumentally, so
to say, objectwise. Functors are certain maps defined on classes of
morphisms, and then specialized to the classes of objects viewed as
identity morphisms. Because the category of convergences with continuous
maps as morphisms is concrete over the category of sets, it is enough
to define concrete functors merely on objects. }).
\begin{defn*}
A concrete, increasing, idempotent and descending functor $J$ is
called a (\emph{concrete}) \emph{reflector.} Then the class of convergences
$\xi$ fulfilling $J\xi=\xi$ is called \emph{reflective}.
\end{defn*}
Topologies an pretopologies are reflective; the reflector $\mathrm{T}$
on the class of topologies is called the \emph{topologizer}, the reflector
$\mathrm{S}_{0}$ on the class of pretopologies is called the \emph{pretopologizer}.
Both admit similar explicit descriptions
\[
\lm_{\mathrm{T}\theta}\mathcal{F}=\bigcap\nolimits _{H\in\mathcal{F}^{\#}}\cl_{\theta}H,\quad\lm_{\mathrm{S}_{0}\theta}\mathcal{F}=\bigcap\nolimits _{H\in\mathcal{F}^{\#}}\ad_{\theta}H,
\]

The objectwise use of functors associated with various classes of
convergences, constitutes a sort of calculus, enabling to perceive
in a unified way miscellaneous aspects of convergences, hence in particular
of topologies.

\section{Adherence-determined classes of convergences}

Let $\xi$ be a convergence on $X$, and let $\mathcal{A}$ be a family
of subsets of $X$. The \emph{adherence} $\ad_{\xi}\mathcal{A}$ is
defined by
\begin{equation}
\ad_{\xi}\mathcal{A}:=\bigcup\nolimits _{\mathcal{H}\subset\mathcal{A}^{\#}}\lm_{\xi}\mathcal{H}.\tag{adherence}\label{eq:adh_fam}
\end{equation}
In particular, if $A\subset X$, then $\ad_{\xi}A=\ad_{\xi}\{A\}$,
which is the \emph{set-adherence}, already introduced in (\ref{eq:adh}).
It is straightforward that, for each filter $\mathcal{F}$ on $X$,
\[
\ad_{\xi}\mathcal{F}=\bigcup\nolimits _{\mathcal{U}\in\beta\mathcal{F}}\lm_{\xi}\mathcal{U}.
\]
We denote by $\mathbb{F}$ the class of \emph{all} filters, by $\mathbb{F}_{1}$
the class of \emph{countably based }filters, and by $\mathbb{F}_{0}$
the class of \emph{principal }(or \emph{finitely based}) filters.
By convention, if $\mathbb{F}_{0}\subset\mathbb{H}\subset\mathbb{F}$
then $\mathbb{H}X$ is the set of filters on $X$ that belong to the
class $\mathbb{H}$ (\footnote{In particular, the filters from $\mathbb{F}_{0}X$ are of the form
$A^{\uparrow}:=\{F\subset X:A\subset F\}$ for some $A\subset X$.}). 

Let $\mathbb{H}$ be a class of filters. We say that $\mathbb{H}$
is \emph{initial} if $f^{-}[\mathcal{H}]\in\mathbb{H}$ for each $\mathcal{H}\in\mathbb{H}$,
final if $f[\mathcal{H}]\in\mathbb{H}$ for each $\mathcal{H}\in\mathbb{H}$
(\footnote{Of course, it is understood that if $f\in Y^{X}$ and $\mathcal{H}\in\mathbb{H}Y$
then $f^{-}[\mathcal{H}]\in\mathbb{H}X$, and if $\mathcal{H}\in\mathbb{H}X$
then $f[\mathcal{H}]\in\mathbb{H}Y$.}).

Assume that $\mathbb{H}$ is an initial class of filters. Then a convergence
$\xi$ is called $\mathbb{H}$\emph{-adherence-determined} if $\lm_{\xi}\mathcal{F}\supset\bigcap\nolimits _{\mathbb{H}\ni\mathcal{H}\subset\mathcal{F}^{\#}}\ad_{\xi}\mathcal{H}.$
The class of $\mathbb{H}$\emph{-}adherence-determined convergences
is \emph{concretely reflective}, and the reflector $\text{\ensuremath{A_{\mathbb{H}}}}$
fulfills
\[
\lm_{A_{\mathbb{H}}\xi}\mathcal{F}=\bigcap\nolimits _{\mathbb{H}\ni\mathcal{H}\subset\mathcal{F}^{\#}}\ad_{\xi}\mathcal{H},
\]
that is, $x\in\lm_{A_{\mathbb{H}}\xi}\mathcal{F}$ provided that $x\in\ad_{\xi}\mathcal{H}$
for each $\mathcal{H}\in\mathbb{H}$ such that $\mathcal{H}\subset\mathcal{F}^{\#}$.

\begin{figure}[h]
\begin{centering}
\includegraphics[scale=0.4]{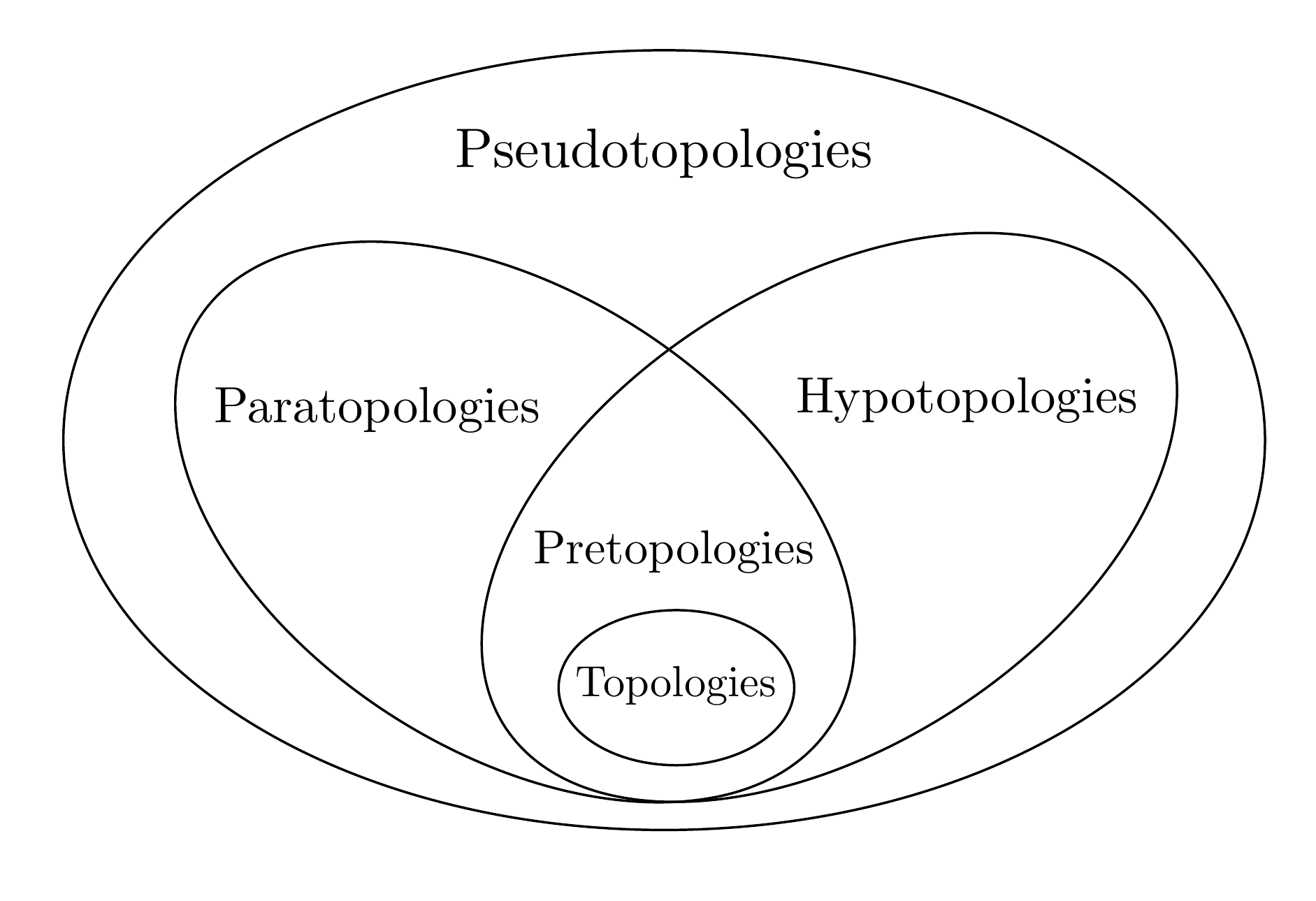}
\par\end{centering}
\caption{Fundamental classes of adherence-determined convergences}
\end{figure}

Since we assume that $\mathbb{F}_{0}\subset\mathbb{H}\subset\mathbb{F}_{1}$,
the $\mathbb{F}_{0}$-adherence-determined \emph{pretopologies} form
the narrowest class, and $\mathbb{F}$-adherence-determined \emph{pseudotopologies}
the largest. By the way, \emph{topologies} are not adherence-determined.

We shall yet consider two intermediate classes: \emph{paratopologies},
corresponding to for which $\mathrm{S}_{1}=A_{\mathbb{F}_{1}}$, and
\emph{hypotopologies}, corresponding to the class $\mathbb{F}_{\wedge1}$
of \emph{countably deep} filters (\footnote{That is, the filters $\mathcal{F}$ such that $\mathcal{F}_{0}\subset\mathcal{F}$
and $\card\mathcal{F}_{0}$ is countable, then $\bigcap\mathcal{F}_{0}\in\mathcal{F}.$ }), for which $\mathrm{S}_{\wedge1}=A_{\mathbb{F}_{\wedge1}}$. It
was recently observed \cite{FM_2020} that pretopologies constitute
the intersection of paratopologies an hypotopologies.

By the way, as all functors, the adherence-determined reflectors preserve
continuity, but moreover fulfill
\begin{equation}
A_{\mathbb{H}}(f^{-}\tau)=f^{-}(A_{\mathbb{H}}\tau)\label{eq:Sf-}
\end{equation}
for each $f$ and $\tau$. In particular, if $f$ is an injection,
then (\ref{eq:Sf-}) means that $A_{\mathbb{H}}$ commutes with the
construction of \emph{subspaces} for $\mathbb{F}_{0}\subset\mathbb{H}\subset\mathbb{F}_{1}$!
Mind that the \emph{topologizer} $\mathrm{T}$ is not of the form
$A_{\mathbb{H}}$, and does not commute with the construction of subspaces.

\section{Pseudotopologies}

We have seen that \emph{pseudotopologies} constitute an adherence-determined
class with respect to the class $\mathbb{F}$ of all filters. It easily
follows from the definition that if $\xi$ is a pseudotopology, then
\begin{equation}
\lm_{\mathrm{S}\xi}\mathcal{F}=\bigcap\nolimits _{\mathcal{U}\in\beta\mathcal{F}}\lm_{\xi}\mathcal{U},\tag{pseudotopologizer}\label{eq:pseudo}
\end{equation}
where $\beta\mathcal{F}$ stands for the set of ultrafilters that
are finer than a filter $\mathcal{F}$. It is remarkable that the
\emph{pseudotopologizer} $\mathrm{S}$ commutes with arbitrary products,
that is, if $\Xi$ is a set of convergences, then
\begin{equation}
\mathrm{S}(\prod\Xi)=\prod_{\xi\in\Xi}\nolimits\mathrm{S}\xi.\tag{commutation}\label{eq:commutation}
\end{equation}
This is because $\mathrm{S}$ commutes with construction of initial
convergence (as an adherence-determined reflector), and also with
arbitrary suprema (\footnote{Let us first prove that $\mathrm{S}(\bigvee\Xi)=\bigvee\nolimits _{\xi\in\Xi}\mathrm{S}\xi.$
Indeed, $\lm_{\mathrm{S}(\bigvee\Xi)}\mathcal{F}=\bigcap\nolimits _{\mathcal{U}\in\beta\mathcal{F}}\lm_{\bigvee\Xi}\mathcal{U}$,
which is equal to $=\bigcap\nolimits _{\mathcal{U}\in\beta\mathcal{F}}\bigcap_{\xi\in\Xi}\lm_{\xi}\mathcal{U}$.
By commuting the intersections, we get $\bigcap_{\xi\in\Xi}\bigcap\nolimits _{\mathcal{U}\in\beta\mathcal{F}}\lm_{\xi}\mathcal{U=}\bigcap_{\xi\in\Xi}\lm_{\mathrm{S}\xi}\mathcal{U}=\lm_{\bigvee_{\xi\in\Xi}\mathrm{S}\xi}\mathcal{F}$.

By the definition of product, $\prod\Xi=\bigvee\nolimits _{\xi\in\Xi}p_{\xi}^{-}\xi$,
where $p_{\xi}:\prod_{\zeta\in\Xi}\left|\zeta\right|\longrightarrow\left|\xi\right|$
is the $\xi$-projection. Therefore, by (\ref{eq:Sf-}), $\mathrm{S}(\prod\Xi)=\mathrm{S}(\bigvee\nolimits _{\xi\in\Xi}p_{\xi}^{-}\xi)=\bigvee\nolimits _{\xi\in\Xi}\mathrm{S}(p_{\xi}^{-}\xi)=\bigvee\nolimits _{\xi\in\Xi}p_{\xi}^{-}(\mathrm{S}\xi)=\prod\nolimits _{\xi\in\Xi}\mathrm{S}\xi$.}).

Although $F(\xi_{0}\times\xi_{1})\geq F\xi_{0}\times F\xi_{1}$ holds
for each functor $F$, the converse is rather an exception. For instance,
the \emph{pretopologizer} $\mathrm{S}_{0}$ commutes with the construction
of initial convergence (like the \emph{pseudotopologizer} $\mathrm{S}$),
but not with suprema, hence not with products.

\section{Quotient maps}

A first example is that of quotient maps. In topology, a continuous
map $f\in C(\xi,\tau)$ (between topologies $\xi$ and $\tau$) is
said to be \emph{quotient} if $\tau\geq\mathrm{T}(f\xi)$, hence,
because of the continuity assumption, $\tau=\mathrm{T}(f\xi)$. Let
us remark that, if $\xi$ is a topology, that is, $\mathrm{T}\xi=\xi$,
then the final convergence $f\xi$ need not be a topology; actually,
it can be, so to say, almost anything, as each \emph{finitely deep}
convergence (\footnote{A convergence $\theta$ is called \emph{finitely deep} if $\lm_{\theta}\mathcal{F}_{0}\cap\lm_{\theta}\mathcal{F}_{1}\subset\lm_{\theta}(\mathcal{F}_{0}\cap\mathcal{F}_{1})$
for any $\mathcal{F}_{0}$ and $\mathcal{F}_{1}$.}) is a convergence quotient of topologies.

It has long been known that quotient maps preserve some properties,
like \emph{sequentiality}, but do not preserve others, like \emph{Fréchetness}.
For this reason, numerous quotient-like maps (\emph{quotient, hereditarily
quotient, countably biquotient, biquotient, triquotient, almost open})
and their preservation properties were intensively investigated. In
his \cite{quest}, \noun{E. Michael} gathered, generalized, and refined
numerous preservation existent results \noun{(A. V. Arhangel'skii}
\cite{arh.factor}, \noun{V. I. Ponomarev} \cite{Pon}, \noun{S. Hanai}
\cite{Han}, \noun{F. Siwiec} \cite{Siwiec}, and others) for these
quotient-like maps. Richness and complexity of these investigations
made of this quotient quest a veritable jungle (\footnote{A metaphor came, when I tackled to decorticate this article. I realized
that I would not grasp its underlying ideas, unless I transform the
jungle into an Italian garden. I evoked it during a conference in
honor of \noun{Peter Collins} and \noun{Mike Reed} in Oxford in 2006,
and \noun{Ernest Michael}, who attended, appreciated.}).

\begin{figure}[h]
\begin{centering}
\includegraphics[scale=0.2]{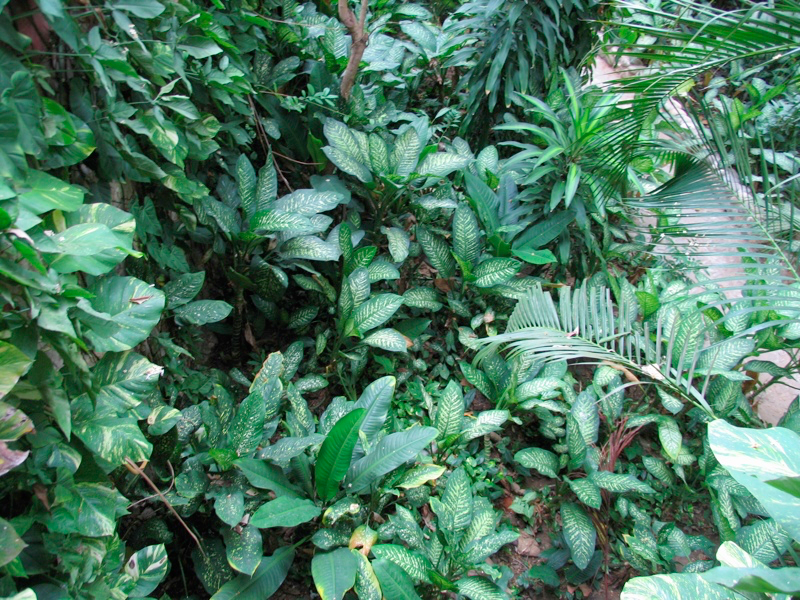}$\quad$\includegraphics[scale=0.15]{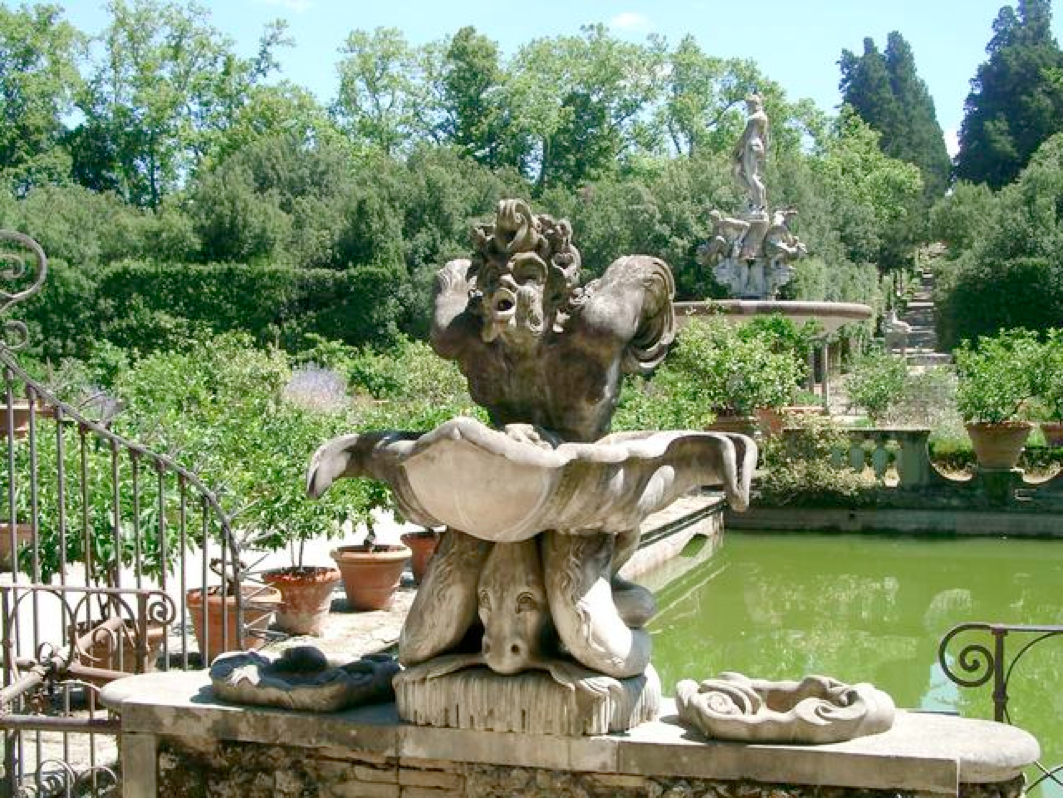}\medskip{}
\par\end{centering}
\centering{}\caption{A multiple quotient quest: transforming a jungle into an Italian garden.}
\end{figure}

Using convergence-theoretic methods \cite{quest2}, it was possible
to figure out that virtually all these quotient-like maps follow the
same pattern, namely they are of the form
\begin{equation}
\tau\geq J(f\xi),\tag{\ensuremath{J}-quotient map}\label{eq:quotient}
\end{equation}
where $J$ is a reflector on a subclass of convergences, for a map
$f\in C(\xi,\tau)$ (\footnote{In fact, in various problems the continuity of a quotient-like map
is inessential, and can be dropped.}), a panorama that was unavailable within the framework of topologies.
In particular, a map $f$ fulfilling (\ref{eq:quotient}) is \emph{quotient}
if $J=\mathrm{T}$, \emph{hereditarily quotient} if $J=\mathrm{S_{0}}$
(pretopologizer), \emph{countably biquotient} if $J=\mathrm{S_{1}}$
(\emph{paratopologizer}), \emph{biquotient} if $J=\mathrm{S}$ (\emph{pseudotopologizer}),
and \emph{almost open} if $J=\mathrm{I}$ (\emph{identity functor}).
By the way, it is often handy to say $\mathbb{H}$\emph{-quotient}
instead of $A_{\mathbb{H}}$-quotient\emph{.}

\emph{Biquotient maps} are the only among the listed classes that
are preserved by arbitrary products, which, of course, is due to (\ref{eq:commutation}).

Of course, (\ref{eq:quotient}) transcends topologies, but when limited
to topologies $\xi$ and $\tau$, it yields a topological conclusion,
having passed beyond. The name \emph{hereditarily quotient,} traditonally
used in the topological context, is due to the fact that each restriction
of a $\mathrm{S}_{0}$-quotient map remains a $\mathrm{S}_{0}$-quotient
(\footnote{Indeed, let $\xi$ be a convergence on $X$, and $\tau$ be a convergence
on $Y$. If $f\in Y^{X}$ fulfills $\tau\geq\mathrm{S}_{0}(f\xi)$,
then for $B\subset Y$ and the injection $j_{B}\in Y^{B}$,
\[
j_{B}^{-}\tau\geq j_{B}^{-}\mathrm{S}_{0}(f\xi)=\mathrm{S}_{0}(j_{B}^{-}\circ f)\xi,
\]
and the final convergence of $\xi$ by $j_{B}^{-}\circ f$ is equal
to the final convergence of $j_{f^{-}(B)}^{-}\xi$ by $j_{B}^{-}\circ f$.}).

It turns out that sundry properties, like \emph{sequentiality, Fréchetness},
\emph{local compactness}, and so on, appear as solutions $\theta$
of \emph{functorial inequalities} of the type
\begin{equation}
\theta\geq JE\theta,\tag{\ensuremath{JE}-property}\label{eq:mixed}
\end{equation}
where $J$ is a concrete reflector, and $E$ is an appropriate concrete
\emph{coreflector}, that is a concrete, increasing, idempotent and
\emph{ascending} ($\theta\leq E\theta$) functor (compare with the
definition of \emph{concrete reflector}).
\begin{example*}
A topology is called \emph{sequential} if each sequentially closed
set is closed. If $\xi$ is a topology, then $\mathrm{Seq\,}\xi$
is the coarsest sequential convergence, in general non-topological,
that is finer than $\xi$. Then $\mathrm{T\,Seq}\,\xi$ stands for
the topology, for which the open sets and the closed sets are determined
by sequential filters, that is, are sequentially open and closed,
respectively. Therefore, a topology $\xi$ is sequential if it coincides
with $\mathrm{T\,Seq}\,\xi$, which is equivalent to
\[
\xi\geq\mathrm{T\,Seq}\,\xi.
\]
A convergence $\xi$ is called \emph{Fréchet} if $x\in\ad_{\xi}A$
implies the existence of a sequential filter $\mathcal{E}$ such that
$A\in\mathcal{E}$ and $x\in\lm_{\xi}\mathcal{E}$. It is straightforward
that $\xi$ is Fréchet whenever
\[
\xi\geq\mathrm{S_{0}\,Seq}\,\xi,
\]
where $\mathrm{S}_{0}$ is the pretopologizer.
\end{example*}
Now a preservation scheme becomes manifest (\footnote{Let $\xi\geq JE\xi$, (\ref{eq:quotient}) and $f\in C(\xi,\tau)$.
Then $f\xi\geq f(JE\xi)\geq JE(f\xi)$, the last inequality being
consequence of $f(F\xi)\geq F(f\xi$), valid for each functor $F$.
By (\ref{eq:quotient}) and idempotency of $J$, we infer $\tau\geq J(f\xi)\geq JE(f\xi)\geq JE\tau$,
the last inequality entailed by continuity: $f\xi\geq\tau$. As a
result, $\tau\geq JE\tau$.}).
\begin{thm*}
If $\xi$ has \emph{(}\ref{eq:mixed}\emph{),} and $f\in C(\xi,\tau)$
is a \emph{(}\ref{eq:quotient}\emph{)}, then $\tau$ has \emph{(}\ref{eq:mixed}\emph{)}.
\end{thm*}
Let us illustrate the preservation result above by the two properties
discussed in the example. A more exhaustive list of special cases
of this theorem, can be found in \cite{quest2}, where all of 20 entries
correspond to theorems, many of which were demonstrated in numerous
papers. See also \cite[p. 400]{CFT}.
\begin{cor*}
A continuous quotient of a sequential topology is sequential. A continuous
hereditarily quotient of a Fréchet topology is Fréchet.
\end{cor*}

\section{Exponential reflective classes\label{sec:Exponential-reflective-classes}}
\begin{defn*}
A class $\mathbb{\mathbf{J}}$ is called \emph{exponential} if $[\xi,\sigma]\in\mathbf{J}$
for any convergence $\xi$, provided that $\sigma\in\mathbf{J}$,
where

The $\sigma$-\emph{dual convergence} $[\xi,\sigma]$ of $\xi$ is
the coarsest convergence on $C(\xi,\sigma)$, for which the evaluation
map $\mathrm{ev}(x,f)=\left\langle x,f\right\rangle :=f(x)$ is jointly
continuous, that is, $\mathrm{ev}\in C(\xi\times[\xi,\sigma],\sigma)$,
that is,
\[
\xi\times[\xi,\sigma]\geq\mathrm{ev}^{-}\sigma.
\]
\end{defn*}
\begin{thm*}
A reflector is exponential if and only if it commutes with finite
products. 
\end{thm*}
Here is a simple proof of sufficiency (\footnote{Let $\sigma=J\sigma$. By definition,
\begin{equation}
\xi\times[\xi,\sigma]\geq\mathrm{ev}^{-}\sigma,\tag{duality}\label{eq:dual}
\end{equation}
and $[\xi,\sigma]$ is the coarsest convergence, for which the inequality
above holds. If $J$ commutes with finite products then, from (\ref{eq:dual}),
\[
\xi\times J[\xi,\sigma]\geq J\xi\times J[\xi,\sigma]\geq J(\xi\times[\xi,\sigma])\geq J(\mathrm{ev}^{-}\sigma)\geq\mathrm{ev}^{-}(J\sigma)=\mathrm{ev}^{-}\sigma,
\]
 the last inequality following from $F(f\sigma)\geq f^{-}(F\sigma)$,
valid for each functor $F$. As, by assumption, $[\xi,\sigma]$ is
the coarsest convergence fulfilling (\ref{eq:dual}), $J[\xi,\sigma]\geq[\xi,\sigma]$,
hence $J[\xi,\sigma]=[\xi,\sigma]$, because $J$ is a reflector.}).

We understand now why the \emph{Kuratowski convergence} on the hyperspaces
of closed sets, considered by Choquet in \cite{cho}, is not topological
(\footnote{Indeed, if $\tau$ is a topology on a set $X$, then the \emph{(upper)}
\emph{Kuratowski convergence} on the hyperspace consisting of all
$\tau$-closed sets is $[\tau,\$]$, where the \emph{Sierpi\'{n}ski
topology} $\$$ on $\{0,1\}$, the closed sets of which are $\0,\{0\},$
and $\{0,1\}$. The hyperspace can be identified with $C(\tau,\$)$,
the space of continuous from $\tau$ to $\$$. Accordingly, $A$ is
$\tau$-closed if and only if the characteristic function $\chi_{A}\in\{0,1\}^{X}$,
that is, $A=\{x\in X:\chi_{A}(x)=1\}$, fulfills $\chi_{A}\in C(\tau,\$)$.}).

Given any reflector $J$, there exists the least \emph{exponential
reflector} $\mathrm{Epi}^{J}$ such that $J\leq\mathrm{Epi}^{J}$
(\footnote{It follows that $J$ is exponential if and only if $J=\mathrm{Epi}^{J}$.}).
The corresponding least exponential reflective class $\mathrm{fix}(\mathrm{Epi}^{J})$
including $\mathrm{fix}(J)$, is called the \emph{exponential hull}
of $\mathbf{J}$. Duality theory, developed by \noun{F. Mynard} \cite{mynard,Mynard.survey},
and others, allows to characterize exponential hulls.

A construction uses the $\sigma$-dual convergence $[[\xi,\sigma],\sigma]$
of the $\sigma$-dual convergence $[\xi,\sigma]$ of $\xi$, which
is a convergence on $C([\xi,\sigma],\sigma)$. Then $\mathrm{Epi}^{\sigma}\xi=j^{-}[[\xi,\sigma],\sigma]$,
that is, the initial convergence of the $\sigma$-bidual convergence
by the natural injection $j:X\longrightarrow Z^{(Z^{X})}$, which
turns out to be continuous: $j\in C(\xi,[[\xi,\sigma],\sigma])$.
Finally,
\[
\mathrm{Epi}^{J}\xi:=\bigvee\nolimits _{\sigma=J\sigma}\mathrm{Epi}^{\sigma}\xi.
\]
It turns out that the two most important non-topological convergences,
introduced by G.\noun{ Choquet,} are intimately related by duality.
\begin{thm*}
The exponential hull of the class of pretopologies is the class of
pseudotopologies.
\end{thm*}
See also (\footnote{Let us mention that $\mathrm{Epi}^{\mathrm{S}_{0}}\xi=\mathrm{S\xi}=j^{-}[[\xi,\yen],\yen]$,
where $\yen$ is the \emph{Bourdaud pretopology}. The \emph{Bourdaud
pretopology} $\yen$ is defined on $\{0,1,2\}$ by convergence of
ultrafilters as follows
\[
\lm_{\yen}\{0\}^{\uparrow}=\{0,1\},\;\lm_{\text{¥}}\{1\}^{\uparrow}=\{0,1,2\},\;\lm_{\text{¥}}\{2\}^{\uparrow}=\{0,1,2\}.
\]
The exponential hull of topologies is the class of \emph{epitopologies},
defined by \noun{P. Antoine}, and then $\mathrm{Epi}^{\mathrm{T}}\xi=j^{-}[[\xi,\$],\$]$.}). 

\section{Compactness versus cover compactness}

A subset $A$ of a topological space is called \emph{compact} if every
open cover of $A$ admits a finite subcover of $A$, equivalently,
each ultrafilter on $A$ has a limit point in $A$, or else, each
filter on $A$ has an adherence point in $A$. Many authors require
that, besides, the topology be Hausdorff.

A convergence is, in general, not determined by its open sets, and
thus open covers are not an adequate concept in this context. A natural
extension to convergence spaces of the notion of cover is used to
define \emph{cover-compact} sets. The point is that \emph{cover-compactness}
and \emph{filter-compactness} are no longer equivalent for general
convergences. Moreover, it turns out that cover-compactness is not
preserved under continuous maps.
\begin{defn*}
Let $\xi$ be a convergence on $X$. A family $\mathcal{P}$ of subsets
of $X$ is called a \emph{$\xi$-cover} of a set $A$, if $\mathcal{P}\cap\mathcal{F}\neq\varnothing$
for every filter $\mathcal{F}$ such that $A\cap\lm_{\xi}\mathcal{F}\neq\0$. 
\end{defn*}
Specializing the definition above to a topology $\xi$ on a set $X$,
we infer that $\mathcal{P}$ is a $\xi$-cover of $A$ if and only
if $A\subset\bigcup\nolimits _{P\in\mathcal{P}}\ih_{\xi}P$, where
$\ih_{\xi}P:=X\setminus\ad_{\xi}(X\setminus A)$ is the \emph{$\xi$-inherence}
of $P$.

Endowed with this extended concept of cover, we are in a position
to discuss cover-compactness for general convergences.
\begin{defn*}
A set $A$ is said to be $\xi$-\emph{cover-compact} if for each $\xi$-cover
of $A$, there exists a finite $\xi$-subcover of $A$; \emph{$\xi$-compact}
if, for each filter $\mathcal{H}$,
\begin{equation}
A\in\mathcal{H}^{\#}\Longrightarrow\ad_{\xi}\mathcal{H}\cap A\neq\0.\tag{compact set}\label{eq:compact}
\end{equation}
\end{defn*}
The following simple (\footnote{A family $\mathcal{P}$ is not a $\xi$-cover of a set $A$, whenever
there exists a filter $\mathcal{F}$ such that $A\cap\lm_{\xi}\mathcal{F}\neq\varnothing$
and $P\notin\mathcal{F}$ for each $P\in\mathcal{P}$. In other words,
$F\cap P^{c}=F\setminus P\neq\varnothing$ for each $P\in\mathcal{P}$
and each $F\in\mathcal{F}$, equivalently, $\mathcal{P}_{c}\#\mathcal{F}$,
that is, $A\cap\ad_{\xi}\mathcal{P}_{c}\neq\varnothing$.}), but very consequential observation \cite{D.comp} enables to easily
compare the two variants.
\begin{prop*}
A family $\mathcal{P}$ is a $\xi$-cover of $A$ if and only if $\ad_{\xi}\mathcal{P}_{c}\cap A=\0.$ 
\end{prop*}
To this end, we focus on ideal covers. A family of subsets of a given
set is called an \emph{ideal} if
\[
(P_{0}\in\mathcal{P})\wedge(P_{1}\in\mathcal{P})\Longleftrightarrow P_{0}\cup P_{1}\in\mathcal{P}.
\]
Clearly, $\mathcal{P}$ is an ideal of subsets of $X$ if and only
if $\mathcal{P}_{c}$ is a filter on $X$. Passing from arbitrary
covers to ideal covers makes no difference in topology, but does make
in general. By the preceding proposition, on setting $\mathcal{H}=\mathcal{P}_{c}$,
we characterize filter-compactness in terms of ideal covers:
\begin{prop*}
A set $A$ is $\xi$-compact if and only if $A\in\mathcal{P}$ for
each ideal $\xi$-cover $\mathcal{P}$ of $A$.
\end{prop*}
Cover-compactness implies (filter)-compactness for pretopologies.
Indeed, if $\xi$ is a pretopology, and $A$ is $\xi$-cover-compact,
then in particular, for each ideal $\xi$-cover $\mathcal{P}$ of
$A$, there exists a finite $\mathcal{P}_{0}\subset\mathcal{P}$ such
that $A\subset\bigcup\nolimits _{P\in\mathcal{P}_{0}}\ih_{\xi}P\subset\ih_{\xi}\bigcup\nolimits _{P\in\mathcal{P}_{0}}P\subset\bigcup\nolimits _{P\in\mathcal{P}_{0}}P\in\mathcal{P}$,
because $\mathcal{P}$ is an ideal. Hence $A\in\mathcal{P}$, so that
$A$ is $\xi$-compact.

On the other hand, there exist pretopologies, where the two notions
differ \cite[Example IX.11.8]{CFT}.

Moreover, each finite set is $\xi$-compact for any convergence $\xi$
(\footnote{In fact, if a filter $\mathcal{H}$ fulfills $\{x\}\in\mathcal{H}^{\#}$
then $x\in\bigcap\mathcal{H}$, and thus $x\in\ad_{\xi}\mathcal{H},$
equivalently $\{x\}\cap\ad_{\xi}\mathcal{H}\neq\0$.}), but 
\begin{prop*}[\cite{myn.relations,CFT}]
 A pseudotopology, the finite subsets of which are cover-compact,
is a pretopology.
\end{prop*}
\begin{proof}
If $\xi$ is not a pretopology, then there is $x\in\left|\xi\right|$
such that each $\xi$-pavement of $x$ is infinite. Thus if $\mathcal{Q}$
is a $\xi$-cover of $\{x\}$ and $\mathbb{P}$ is a $\xi$-pavement
at $\{x\}$, then $\mathcal{Q}\cap\mathcal{P}\neq\0$ for each $\mathcal{P}\in\mathbb{P}$,
so that $\mathcal{Q}$ cannot be finite.
\end{proof}
\begin{cor*}
Continuous maps between non-pretopological spaces do not preserve
cover-compactness \emph{(}\footnote{If $\xi$ is a convergence on $X$ such that $\mathfrak{p}(x_{0},\xi)$
is infinite, $\iota$ is the discrete topology on $X$, then for the
identity map $i_{X}\in C(\iota,\xi)$, the image $i(\{x_{0}\})$ is
not $\xi$-cover-compact, but $\{x_{0}\}$ is $\iota$\nobreakdash-cover
compact.}\emph{)}.
\end{cor*}

\section{Extensions of the concept of compactness}

\emph{Compact families} of sets generalize both compact sets and convergent
filters, and this generalization is not just a whim. It has important
applications, and, perhaps more importantly, evidences mathematical
laws that remained invisible on the level of compactness of sets.

Let $\xi$ be a convergence on $X$. A family $\mathcal{A}$ of subsets
of $X$ is said to be $\xi$\emph{-compact at} a family $\mathcal{B}$
of subsets of $X$ if, for each filter $\mathcal{H}$,
\begin{equation}
\mathcal{A}\subset\mathcal{H}^{\#}\Longrightarrow\ad_{\xi}\mathcal{H}\in\mathcal{B}^{\#}.\tag{compact family}\label{eq:compact_family}
\end{equation}
In particular, $\mathcal{A}$ is called $\xi$\emph{-compact} if it
is $\xi$-compact at itself; $\xi$\emph{-compactoid} if it is $\xi$-compact
at $X$ (\footnote{The set $\kappa(\xi)$ of all $\xi$-compact (isotone) families on
$X=\left|\xi\right|$ fulfills: $\0,2^{X}\in\kappa(\xi)$, $\{\mathcal{A}_{j}:j\in J\}\subset\kappa(\xi)$
entails $\bigcup_{j\in J}\mathcal{A}_{j}\in\kappa(\xi)$, and $\bigcap_{j\in J}\mathcal{A}_{j}\in\kappa(\xi)$,
whenever $J$ is finite. In other words, $\kappa(\xi)$ has the properties
of a family of open sets of a topology on $2^{X}$.}).

It is clear that a subset $A$ of $X$ is $\xi$-\emph{compact} ($\xi$-\emph{compactoid}),
whenever $A^{\uparrow}:=\{F\subset X:A\subset F\}$ is ($\footnotemark[15]$).
On the other hand, it is straightforward that $\mathcal{F}$ is $\xi$-compact
at $\{x\}$ if and only if $x\in\lm_{\mathrm{S}\xi}\mathcal{F}$.
Incidentally, it is straightforward that $\xi$-compactness and $\mathrm{S}\xi$-compactness
coincide.

This simple fact prefigures the pseudotopological nature of compactness,
which will be evidenced in a moment. 

At this point, it will be instrumental to consider again the notion
of grill, from a somewhat different perspective. Recall that $\mathcal{A}^{\#}:=\bigcap_{A\in\mathcal{A}}\{H\subset X:A\cap H\neq\0\}$
for a family $\mathcal{A}$ of subsets of $X$. Now, for another family
$\mathcal{H}$ on $X$, the condition $\mathcal{H}\subset\mathcal{A}^{\#}$
is equivalent to $\mathcal{A}\subset\mathcal{H}^{\#}$, so we denote
this relation symmetrically, by $\mathcal{H}\#\mathcal{A}$ (\footnote{Of course, $\mathcal{H}\#\mathcal{A}$ whenever $H\cap A\neq\0$ for
each $H\in\mathcal{H}$ and $A\in\mathcal{A}$.}). If $\mathcal{A}$ is on $X$, and $\mathcal{B}$ is on $Y$, and
$f:X\longrightarrow Y$, then it is easy to see that
\begin{equation}
f[\mathcal{A}]\#\mathcal{B}\Longleftrightarrow\mathcal{A}\#f^{-}[\mathcal{B}].\tag{grill}\label{eq:grill}
\end{equation}

For a given convergence $\xi$ on $X$, define the associated \emph{characteristic
convergence} $\chi_{\xi}$ by
\begin{equation}
\lm_{\chi_{\xi}}\mathcal{F}:=\begin{cases}
X & \lm_{\xi}\mathcal{F}\neq\varnothing,\\
\varnothing & \lm_{\xi}\mathcal{F}=\varnothing.
\end{cases}\tag{characteristic}\label{eq:characteristic}
\end{equation}

It is immediate that, for a set $\Xi$ of convergences,
\begin{equation}
\chi_{\prod\Xi}=\prod\nolimits _{\xi\in\Xi}\chi_{\xi}.\tag{characteristic of product}\label{eq:char_prod}
\end{equation}

\begin{lem*}
A filter $\mathcal{F}$ is $\xi$-compactoid if and only if $\lm_{\mathrm{S}\chi_{\xi}}\mathcal{F}\neq\0$.
\end{lem*}
\begin{proof}
By (\ref{eq:pseudo}), $\lm_{\mathrm{S}\chi_{\xi}}\mathcal{F}\neq\0$
if and only if $\lm_{\chi_{\xi}}\mathcal{U}\neq\0$ for each $\mathcal{U}\in\beta\mathcal{F}$,
equivalently, by (\ref{eq:characteristic}), $\lm_{\xi}\mathcal{U}\neq\0$
for each $\mathcal{U}\in\beta\mathcal{F}$.
\end{proof}
As an immediate consequence of this lemma and of (\ref{eq:commutation}),
\begin{thm*}[Generalized Tikhonov Theorem]
 A filter $\mathcal{F}$ is $\prod\Xi$-compactoid if and only if
$p_{\xi}[\mathcal{F}]$ is $\xi$-compactoid for each $\xi\in\Xi$.
\end{thm*}
\begin{proof}
By (\ref{eq:char_prod}) and (\ref{eq:commutation}), $\mathrm{S}(\chi_{\prod\Xi})=\mathrm{S}(\prod\nolimits _{\xi\in\Xi}\chi_{\xi})=\prod\nolimits _{\xi\in\Xi}\mathrm{S}\chi_{\xi}$.
The proof is complete in virtue of Lemma above.
\end{proof}
If we restrict filters $\mathcal{H}$ in (\ref{eq:compact_family})
to a class $\mathbb{H}$ of filters, then we obtain a notion of $\mathbb{H}$\emph{-compactness}.
We assume that $\mathbb{F}_{0}\subset\mathbb{H}\subset\mathbb{F}$,
that is, that the said class includes all principal filters. $\mathcal{A}$
is said to be $\xi$\emph{-$\mathbb{H}$-compact at} $\mathcal{B}$
if
\[
\underset{\mathcal{H}\in\mathbb{H}}{\forall}\;\mathcal{A}\subset\mathcal{H}^{\#}\Longrightarrow\ad_{\xi}\mathcal{H}\in\mathcal{B}^{\#}.
\]

Some instances of this notion have been already known in topological
context, like $\mathbb{F}_{1}$\emph{-compactness}, that is, \emph{countable
compactness }(\footnote{By the way, \emph{sequential compactness} of $\xi$ coincides with
$\mathbb{F}_{1}$-compactness of $\mathrm{Seq}\xi$.}), or $\mathbb{F}_{\wedge1}$\emph{\nobreakdash-compactness}, that
is, \emph{Lindelöf property.} If $H=A_{\mathbb{H}}$ is the $\mathbb{H}$-adherence-determined
reflector, then a filter $\mathcal{F}$ is $\mathbb{H}$-\emph{compactoid}
for $\xi$, whenever (\footnote{Recall that $x\in\lm_{H\xi}\mathcal{F}$ whenever $\mathcal{F}$ is
$\xi$-$H$-compact at $\{x\}$.})
\[
\lm_{H\chi_{\xi}}\mathcal{F}\neq\0.
\]
Of course, once established for special reflectors, the formula above
can be used as definition of $H$-compactness for arbitrary reflectors
$H$.

Observe that, for other refectors $H$ than the pseudotopologizer,
$H$-compactness is nor preserved even by finite products if $H$
does not commute with such products.

\section{Perfect-like maps}

A step further is to extend $\mathbb{H}$-compactness to relations.
Roughly speaking (\footnote{Let $\theta$ be a convergence on $W$ and $\sigma$ be a convergence
on $Z$. A relation $R$ is called $\mathbb{J}$\emph{-compact }if
$w\in\lm_{\theta}\mathcal{F}$ implies that $R(w)\#\ad_{\sigma}\mathcal{H}$
for each $\mathcal{H}\#R[\mathcal{F}]$ such that $\mathcal{H}\in\mathbb{J}$.}), a relation $R$ is $\mathbb{H}$-\emph{compact} if $y\in\lm\mathcal{F}$
implies that $R[\mathcal{F}]$ is $\mathbb{H}$-compact at $Ry$.
Continuous maps and various quotient maps can be characterized in
terms of compact relations \cite{myn.relations}. However, most advantageous
applications of compact relations are to various perfect-like maps.

A surjective map $f$ is $\mathbb{H}$\emph{-perfect} if and only
if the relation $f^{-}$ is $\mathbb{H}$-compact.

For instance, $\mathbb{F}$-\emph{perfect maps} are precisely \emph{perfect
maps} are close maps with compact fibers. $\mathbb{F}_{1}$-\emph{perfect
maps,} or \emph{countably perfect maps} are close maps with countably
compact fibers.
\begin{prop*}
A surjective map $f\in C(\xi,\tau)$ is $\mathbb{H}$\emph{-perfect}
if and only if
\[
f^{-}\tau\geq A_{\mathbb{H}}(\chi_{\xi}).
\]
\end{prop*}
In particular, arbitrary product of $\mathbb{F}$-perfect maps is
$\mathbb{F}$-perfect. This is because the pseudotopologizer $\mathrm{S}=A_{\mathbb{F}}$
commutes with arbitrary products, or else because the product of compact
fiber relations is compact by the \emph{Generalized Tikhonov Theorem}.

Perfect-like and quotient-like properties embody various degrees of
converse stability of maps $f$, or in other terms, of stability of
fiber relations $f^{-}$, which is the inverse relation of $f$. Let
us rewrite these properties in expanded form, where $f:\left|\xi\right|\longrightarrow\left|\tau\right|$.

A surjective map $f$ is $\mathbb{H}$\emph{-quotient} if and only
if

\begin{equation}
f^{-}(\ad_{\tau}\mathcal{H})\subset\ad_{\xi}f^{-}[\mathcal{H}]\label{eq:quot}
\end{equation}
holds for each $\mathcal{H}\in\mathbb{H}$. A surjective map $f$
is $\mathbb{G}$\emph{-perfect} if and only if
\begin{equation}
\ad_{\tau}f[\mathcal{G}]\subset f(\ad_{\xi}\mathcal{G})\label{eq:per}
\end{equation}
holds for each $\mathcal{G}\in\mathbb{G}$.
\begin{lem*}
Let $\mathbb{F}_{0}\subset\mathbb{A}\subset\mathbb{F}$ be such that
$\mathcal{A}\in\mathbb{A}$ implies $f^{-}[\mathcal{A}]\in\mathbb{A}$
and $f[\mathcal{A}]\in\mathbb{A}$. Then every $\mathbb{A}$-perfect
map is $\mathbb{A}$-quotient.
\end{lem*}
\begin{proof}
Set $\mathcal{G}:=f^{-}[\mathcal{H}]$ and apply $f^{-}$ to both
sides of (\ref{eq:per}). As $f[f^{-}[\mathcal{H}]]=\mathcal{H}$,
because $f$ is surjective, and since by $f^{-}(f(H))\subset H$ for
each $H$, we get
\[
f^{-}(\ad_{\tau}\mathcal{H})\subset f^{-}(\ad_{\tau}f[f^{-}[\mathcal{H}]])\subset f^{-}(f(\ad_{\xi}f^{-}[\mathcal{H}]))\subset\ad_{\xi}f^{-}[\mathcal{H}].
\]
\end{proof}
\begin{table}[H]
\begin{centering}
\begin{tabular}{|c|c|c|c|}
\hline 
perfect-like & $\Rightarrow$ & quotient-like & reflector\tabularnewline
\hline 
\hline 
 &  & open & \tabularnewline
\hline 
 &  & almost open & identity $\mathrm{I}$\tabularnewline
\hline 
perfect & $\Rightarrow$ & biquotient & pseudotopologizer $\mathrm{S}$\tabularnewline
\hline 
countably perfect & $\Rightarrow$ & countably biquotient & paratopologizer $\mathrm{S}_{1}$\tabularnewline
\hline 
adherent & $\Rightarrow$ & hereditarily quotient & pretopologizer $\mathrm{S}_{0}$\tabularnewline
\hline 
closed & $\Rightarrow$ & topologically quotient & topologizer $\mathrm{T}$\tabularnewline
\hline 
\end{tabular}
\par\end{centering}
\medskip{}

\caption{Interrelations between perfect-like maps, quotient-like maps and reflective
classes.}
\end{table}

No arrow can be reversed. Indeed,
\begin{example*}
\label{exa:classic}Let $f:\mathbb{R}\rightarrow S^{1}$ be given
by $f(x):=(\cos2\pi x,\sin2\pi x)$. It follows immediately from the
proposition above that $f$ is open, hence, has all the properties
from the right-hand column. Notice that the the set $\{n+\frac{1}{n}:n\in\mathbb{N}\}$
is closed, but its image by $f$ is not closed, so that $f$ is not
closed, and thus has no property from the left-hand column.
\end{example*}

\section{Conclusions}

I hope that these outlines allow to grasp the essence of convergence
theory. Sure enough, only few aspects have been touched upon, and
most remain beyond this presentation.

For example, various types of compactness are instances of numerous
kinds of completeness. The completeness number $\mathrm{compl}(\xi)$
of a convergence $\xi$ is the least cardinality of a collection of
$\xi$-non-adherent filters that fill the set of $\xi$-non-convergent
ultrafilters in the Stone space. This way, \emph{compact convergences}
$\xi$ are characterized by $\mathrm{compl}(\xi)=0,$ \emph{locally
compactoid} by $\mathrm{compl}(\xi)<\infty$, and \emph{topologically
complete} by $\mathrm{compl}(\xi)\leq\aleph_{0}$. Each convergence
has its completeness number; for the \textquotedblleft very incomplete\textquotedblright{}
space of \emph{rational numbers}, this number is the \emph{dominating
numbe}r $\mathfrak{d}$.

It was shown in this paper that a generalization of \emph{Tikhonov
Theorem} is a simple corollary of the commutation of the pseudotopologizer
with arbitrary products. It turns out that it is also a simple consequence
of a theorem on the completeness number of products \cite{compl_number}\cite{CFT}.

It appears that $\mathrm{compl}(\xi)$ is equal to the \emph{(free)
pseudo-paving number} of the dual convergence $[\xi,\$]$ at $\0$,
and the \emph{(free) paving number }of $[\xi,\$]$ at $\0$ is equal
to the ultra-completeness number of $\xi$ \cite{FM_ultra}.

We see that, in the framework of topologies, it would be impossible
to consider a property dual to \emph{\v{C}ech completeness} (countable
completeness number), because the paving and pseudo-paving numbers
of a topology do not exceed $1$.

I could long display similar examples, but I expect that these few
exhibited in this paper would convince you of the interest of convergence
theory.

\bibliographystyle{plain}
\bibliography{biblio2015}

\end{document}